\documentclass[11pt]{amsart} \usepackage{latexsym, amssymb, stmaryrd, amsthm, amscd, amsmath}
\usepackage[T1]{fontenc}

\newtheorem{thm}{Theorem}[section]
\newtheorem{lemma}[thm]{Lemma}
\newtheorem{prop}[thm]{Proposition}
\newtheorem{cor}[thm]{Corollary}
\newtheorem{question}[thm]{Question}
\newtheorem{fact}[thm]{Fact}

\theoremstyle{definition}
\newtheorem{df}[thm]{Definition}
\newtheorem{nrmk}[thm]{Remark}
\newtheorem{notation}[thm]{Notation}

\theoremstyle{remark}

\makeatletter

\def\dotminussym#1#2{%
  \setbox0=\hbox{$\m@th#1-$}%
  \kern.5\wd0%
  \hbox to 0pt{\hss\hbox{$\m@th#1-$}\hss}%
  \raise.6\ht0\hbox to 0pt{\hss$\m@th#1.$\hss}%
  \kern.5\wd0}
\newcommand{\dotminus}{\mathbin{\mathpalette\dotminussym{}}}
\renewcommand{\r}{\mathbb{R}}
\newcommand{\Z}{\mathbb{Z}}

\newcommand{\curly}[1]{\mathcal{#1}}

\newcommand{\B}{\curly{B}}

\newcommand{\n}{\mathbb{N}}

\newcommand{\la}{\curly{L}}

\renewcommand{\to}{\rightarrow}

\newcommand{\e}{\epsilon}
\newcommand{\de}{\delta}

\def \f{\mathbb F}
\def \logic{\operatorname{logic}}
\def \wk{\operatorname{wk}}

\def \<{\langle}
\def \>{\rangle}

\def \*Z {{{^*}\Z}}
\def \((  {(\!(}
\def \)) {)\!)}

\def \O{\operatorname{O}}

\numberwithin{equation}{section}

\def \u{\mathcal U}

\def \O{\mathcal O}

\def \id{\operatorname{id}}
\def\cb{\operatorname{cb}}
\def\om{\omega}
\def\bF{\mathbb F}
\def\cS{\mathcal S}
\def \span{\operatorname{span}}

\newcommand{\tom}[1]{\textbf{\Large !!}{\footnotesize[[TS- #1]]}}

\allowdisplaybreaks[2]

\DeclareMathOperator{\OS}{\mathcal{OS}}
\DeclareMathOperator{\OSy}{\mathcal{OS}y}
\DeclareMathOperator{\Ex}{\mathcal E}

\DeclareMathOperator{\X}{\mathcal{X}}
\DeclareMathOperator{\EX}{\mathcal{EX}}

\title{Omitting types in operator systems}
\author{Isaac Goldbring and Thomas Sinclair}
\thanks{Goldbring's work was partially supported by NSF CAREER grant DMS-1349399. Sinclair was partly supported by an NSF RTG Assistant Adjunct Professorship.}

\address {Department of Mathematics, Statistics, and Computer Science, University of Illinois at Chicago, Science and Engineering Offices M/C 249, 851 S. Morgan St., Chicago, IL, 60607-7045}
\email{isaac@math.uic.edu}
\urladdr{http://www.math.uic.edu/~isaac}

\address{Department of Mathematics, Indiana University, Rawles Hall, 831 E 3rd St, Bloomington, IN 47405}
\email{tsinclai@indiana.edu}

\begin{document}

\begin{abstract}
We show that the class of 1-exact operator systems is not uniformly definable by a sequence of types.  We use this fact to show that there is no finitary version of Arveson's extension theorem.  Next, we show that WEP is equivalent to a certain notion of existential closedness for C$^*$-algebras and use this equivalence to give a simpler proof of Kavruk's result that WEP is equivalent to the complete tight Riesz interpolation property.  We then introduce a variant of the space of n-dimensional operator systems and connect this new space to the Kirchberg Embedding Problem, which asks whether every C$^*$-algebra embeds into an ultrapower of the Cuntz algebra $\mathcal{O}_2$.  We end with some results concerning the question of whether or not the local lifting property (in the sense of Kirchberg) is uniformly definable by a sequence of types in the language of C$^*$-algebras.
\end{abstract}

\maketitle

\tableofcontents

\section{Introduction}

It has recently been observed that many properties of C$^*$-algebras, while not axiomatizable (in the sense of first-order logic), are \emph{uniformly definable by a sequence of types}, that is, for each $(m,j)\in \n^2$, there is a formula $\phi_{m,j}(\vec x_m)$ in the language of C$^*$-algebras such that:

\begin{enumerate}
\item each $\phi_{m,j}$ takes only nonnegative values;
\item  for a fixed $m$, each $\phi_{m,j}$ has the same modulus of uniform continuity; 
\item  a C$^*$-algebra $A$ has the property if and only if, for each $m\in \n$, we have $$\left(\sup_{\vec x_m}\inf_j \phi_{m,j}(\vec x_m)\right)^A=0.$$  
\end{enumerate}
In other words, the C$^*$-algebra has the property if and only if it omits each of the types $\Gamma_{m,n}(\vec x_m):=\{\phi_{m,j}(\vec x_m)\geq \frac{1}{n}\ : \ j\in \n\}$.  Condition (1) is merely a convenient normalization.  Condition (2) is used in technical applications:  it ensures that the infinitary formula $\inf_j \phi_{m,j}$ is uniformly continuous, that its interpretation has the same modulus of uniform continuity when interpreted in any C$^*$-algebra, and it is also crucial in the proof of the Omitting Types Theorem for uniform sequences of types (see \cite[Theorem 4.2]{FM}). 

In \cite{FHRTW}, it is proved that the following properties of a C$^*$-algebra are uniformly definable by a sequence of types:  UHF (for separable algebras), AF (again, for separable algebras), nuclear, nuclear dimension $\leq n$ (for any given $n$), decomposition rank $\leq n$ (again, for any given $n$), Popa, TAF, simple (for unital C$^*$-algebras).

In connection with the results from \cite{FHRTW}, Ilijas Farah asked whether or not exactness is uniformly definable by a sequence of types.  In this paper, we answer a related question by proving that 1-exactness for operator systems is \emph{not} uniformly definable by a sequence of types (in the language of operator systems).  (We should remark that the types defining nuclearity are types in the language of operator systems.)  The proof proceeds by showing that, if the $1$-exact operator systems were uniformly definable by a sequence of types, then, for any $n$, the $n$-dimensional $1$-exact operator systems form a dense $G_\de$ subset of the space $\OS_n$ of all $n$-dimensional operator systems in the weak topology defined by Junge and Pisier in \cite{jungepisier}.  We then observe that the arguments in \cite{jungepisier} show that the $1$-exact operator systems are not $G_\delta$.

We use the fact that the 1-exact operator systems are not uniformly definable by a sequence of types to prove that there cannot exist a quantitative, finitary version of Arveson's extension theorem.  We should mention that before we proved that exactness is not uniformly definable by a sequence of types, Taka Ozawa outlined a proof that the quantitative version of Arveson's extension theorem cannot hold.  We include his proof here, although our proof is technologically more elementary.  

After connecting finitary Arveson extension with the Weak Expectation Property (WEP), we take the opportunity here to answer a question from our earlier paper \cite{GS} by proving that WEP is equivalent to semi-positive existential closedness as an operator system.  We use this equivalence to give a simpler proof of a result of Kavruk \cite{kav-riesz}, namely that WEP is equivalent to the complete tight Riesz interpolation property.

We then describe a variant of the spaces $\OS_n$ that we then use to give a new equivalent formulation of the \emph{Kirchberg embedding problem}, which asks whether every C$^*$-algebra embeds into an ultrapower of the Cuntz algebra $\mathcal{O}_2$.

In the final section, we show how our results imply that the class of operator systems that have the local lifting property (in the sense of Kirchberg) is not uniformly definable by a sequence of types.  In a positive direction, we show that the class of C$^*$-algebras with the local lifting property is infinitarily axiomatizable.  We end with some results motivating the need to settle the question of whether or not the local lifting property is uniformly definable by a sequence of types.

We will assume that the reader is familiar with the operator space and system notions being described in this paper, although we will occasionally remind the reader of the definitions of these notions.  We will also assume that the reader is familiar with the model-theoretic treatment of operator algebras, spaces, and systems; our earlier paper \cite{GS} describes this in great detail.  That being said, we include appendices containing two model-theoretic discussions that are not as widely known.

We would like to thank Bradd Hart and Taka Ozawa for useful conversations regarding this work.

\section{Exact operator spaces and systems}

Let $\OS_n$ be the space of all isomorphism classes of $n$-dimensional operator spaces. There are two natural (metric) topologies on this space which we will refer to as the \emph{strong} and \emph{weak} topologies. For $E,F\in \OS_n$ and $k\geq 1$, we define \[d_k(E,F):=\inf\{\|u\|_k\cdot \|u^{-1}\|_k \ : \ u:E\to F \text{ linear bijection}\}.\]  We define $d_{\cb}(E,F):=\sup_{k\geq 1}d_k(E,F)$.

One can show that $\log d_{\cb}$ is a (complete) metric on $\OS_n$ and the resulting topology is called the \emph{strong topology} on $\OS_n$.  It is straightforward to verify that a net $E_i$ converges to $E$  strongly if and only if there are linear bijections $\phi_i:E_i\to E$ so that $\|\phi_i\|_{\cb}\cdot \|\phi_i^{-1}\|_{\cb}\to 1$.

Similarly, $\log d_k$ is a metric on $\OS_n$ as is the metric $\delta_w:=\sum_{k\geq 1} 2^{-k}\log d_k$.  The topology induced by $\delta_w$ is called the \emph{weak topology} on $\OS_n$.  Here, a net $E_i$ converges to $E$ weakly if and only if there is a net of linear bijections $\phi_i: E_i\to E$ so that, for all $k$, we have $\|\phi_i\|_k,\|\phi_i^{-1}\|_k\to 1$; equivalently, $E$ is completely isomorphic to the ultraproduct $\prod_\u E_i$ for any nonprincipal ultrafilter $\u$ on the index set.  It follows that $(\OS_n,\wk)$ is a \emph{compact} Polish space.

Recall that a finite-dimensional operator space $E$ is said to be \emph{1-exact} if there are linear injections $\phi_i:E\to M_{n_i}$ such that $\|\phi_i\|_{\cb}\cdot \|\phi_i^{-1}\|_{\cb}\to 1$.  We let $\Ex_n$ denote the set of $n$-dimensional 1-exact operator spaces.  It follows that $\Ex_n$ is the strong closure of the $n$-dimensional matricial operator spaces.
%

As we will see in the next section, it will be important to understand the complexity (in the sense of descriptive set theory) of $\Ex_n$ in the weak topology.  Our first observation is that, for $n\geq 3$, $\Ex_n$ is not weakly comeager.  In fact, this is precisely what Junge and Pisier \cite[Theorem 2.3]{jungepisier} establish in order to conclude, via Baire Category arguments, that $\OS_n$ is not strongly separable for $n\geq 3$.

The proof of the following theorem is essentially established in \cite{jungepisier} although it is not explicitly stated there as such.
\begin{thm}\label{notGde}
For any $n\geq 3$, $\Ex_n$ is not weakly comeager. 
\end{thm}

\begin{proof}
Suppose, towards a contradiction, that $\Ex_n$ is weakly comeager.    By \cite[Proposition 2.1]{jungepisier} and \cite[Theorem 2.2]{ER}, the map $E\mapsto E^*$ (the operator space dual) on $\OS_n$ is an isometric bijection (in either the weak or strong topology).  It follows that $\{E \in \OS_n \ : \ E^* \in \Ex_n\}$ is also comeager, so by the Baire Category Theorem, $\{E\in \OS_n \ : E,E^*\in \Ex_n\}$ is dense, which is precisely the fact that leads to a contradiction in \cite[Theorem 2.3]{jungepisier}.
\end{proof}

As $\Ex_n$ is weakly dense (see, for example, the proof of \cite[Proposition 2.2]{jungepisier}), this shows:

\begin{cor}  For any $n\geq 3$, $\Ex_n$ is not weakly $G_\de$.
\end{cor}


\begin{nrmk} 
We achieved a contradiction in the previous corollary by showing that, for $n\geq 3$, if $\Ex_n$ is weakly $G_\delta$, then $\OS_n$ is strongly separable.  In \cite{jungepisier}, Junge and Pisier prove that if $\OS_n$ is strongly separable, then in turn any strongly closed set (e.g. $\Ex_n$) is weakly $G_\delta$.  Thus, any proof that $\Ex_n$ is not weakly $G_\delta$ would yield a new proof that $\OS_n$ is not strongly separable. 
\end{nrmk}

\begin{nrmk}
For any $n\geq 3$, $(\Ex_n,\wk)$ is not a Baire space.  Indeed, suppose that $(\Ex_n,\wk)$ is a Baire space, and consider the identity map $\operatorname{id}:(\Ex_n,\wk)\to(\Ex_n,\operatorname{strong})$.   By the same argument as in \cite[Theorem 2.3]{jungepisier}), $\id$ is Baire class one.  By \cite[Theorem 8.38]{Kec}, the points of continuity of $\id$ are dense in $(\Ex_n,\wk)$; by \cite[Proposition 2.2]{jungepisier}, the points of continuity of $\id$ are those $E\in \Ex_n$ for which $E^*\in \Ex_n$.  It follows that the set of all $E\in \OS_n$ so that both $E$ and $E^*$ are in $\Ex_n$ is weakly dense in $\OS_n$, leading to the same contradiction.
\end{nrmk}

Although it plays no pivotal role in the sequel, we observe the following:

\begin{lemma}
For every $n$, $\Ex_n$ is a weakly $\Pi_3^0$ subset of $\OS_n$.
\end{lemma}

\begin{proof}
By the small perturbation argument, for any unital separable C$^*$-algebra $A$, the space of all $n$-dimensional operator subsystems of $A$ is separable in the strong topology for all $n$. As the relative weak and strong topologies coincide on the matricial operator systems (Smith's lemma), the set of matricial operator systems is separable in the weak topology; let $(F_k)$ enumerate a countable dense set of them.  Then $E\in \neg \Ex_n$ if and only if $$E\in \bigcup_m \bigcap_k  \bigcup_l \{E' \in \OS_n \ : \ d_l(E',F_k)>1 + \tfrac{1}{m}\}.$$  It remains to note that $\{E'\in \OS_n \ : \ d_l(E',F_k)> 1 +\frac{1}{m}\}$ is weakly open.
\end{proof}

\begin{question}\label{hierarchy}
Where exactly in the (weak) Borel hierarchy does $\Ex_n$ lie?  In particular, is $\Ex_n$ weakly $\Pi_3^0$-complete?
\end{question}

\begin{nrmk}
Let $\OSy_n$ be the subset of  $\OS_n$ consisting of all $n$-dimensional operator systems, where the morphisms are now \emph{unital} linear bijections. We conclude this section by mentioning that all of the discussion in this section holds for the category of $n$-dimensional operator systems rather than the category of $n$-dimensional operator spaces, the only exception being that the corresponding version of Theorem \ref{notGde} holds for $n\geq 5$ instead of $n\geq 3$.
\end{nrmk}

\section{The main result}

In this section, we use the fact that $\Ex_n$ is not $G_\de$ in the weak topology to prove that the 1-exact operator spaces (resp. 1-exact operator systems) are not uniformly definable by a sequence of types in the language of operator spaces (resp. the language of operator systems).  For simplicity, we work entirely in the operator space category, although all proofs carry over to the operator system category without any change to the proofs.

In what follows, we use the notation and terminology from Appendix \ref{codes}.  We let $\la$ denote the language of operator spaces (see, e.g. \cite{GS}) and we let $\frak{M}_n\subseteq \frak{M}$ denote the set of codes for $n$-dimensional operator spaces.  We claim that $\frak{M}_n$ is a $G_\de$ subset of $\frak{M}$, whence $(\frak{M}_n,\logic)$ is also a Polish space.  Indeed, the set of operator systems of dimension at most $n$ is a universally axiomatizable class, whence the set of codes for such operator systems forms a closed subset of $\frak{M}$ by Lemma \ref{closed}.  It follows that $\frak{M}_n$ is the intersection of a closed subset of $\frak{M}$ with an open subset of $\frak{M}$, so it is $G_\de$. 

We have a ``forgetful'' map $F:\frak{M}_n\to \OS_n$ given by sending an element of $\frak{M}_n$ to the operator space it codes.  

\begin{lemma}\label{open}
The map $F:(\frak{M}_n, \logic)\to (\OS_n,\wk)$ is a continuous, open, surjective map.
\end{lemma}

\begin{proof} Continuity is trivial to check as is surjectivity. Fix $E\in \OS_n$ and choose $X\in F^{-1}(E)$. Let $\{x_i\}$ be the coding of $E$ given by $X$. A basic open neighborhood $\mathcal U$ of $X$ checks a finite number of conditions over a finite number of sorts involving only $\{x_1,\dotsc, x_q\}$ for some $q$. It follows that we may choose $k, \eta>0$ so that for each $E'$ such that $d_k(E',E)<\eta$ witnessed by $\phi: E'\to E$, we have that the coding $X'$ of $E'$ defined by $\{x_i' := \phi^{-1}(x_i)\}$ belongs to $\mathcal U$. Hence the basic weak open neighborhood $\{E'\in \OS_n : d_k(E',E)<\eta\}$ of $E$ is contained in $F(\mathcal U)$, and openess follows.

\end{proof}

%
%
%

\begin{thm}\label{notomittingtypes}
The class of $1$-exact operator spaces is not uniformly definable by a sequence of types in the language of operator spaces.  
\end{thm} 

\begin{proof}
Suppose, towards a contradiction, that the class of $1$-exact operator spaces is uniformly definable by a sequence of types in the language of operator spaces.  Then there are formulae $\phi_{m,j}(x_1,\ldots,x_{l_m})$ in the language of operator spaces taking only nonnegative values so that, for an operator space $E$, we have that $E$ is 1-exact if and only if $$\sup_m (\sup_{\vec x}\inf_j \phi_{m,j}(\vec x))^E=0.$$  Fix $n\geq 3$.  We obtain a contradiction by showing that $\Ex_n$ is $G_\de$ in the weak topology on $\OS_n$.  Let $(x_{m,k})_k$ enumerate all elements of $\mathbb{N}^{l_m}$.  For $k,q\geq 1$, set $$U_{m,k,q}:=\{E\in \frak{M}_n \ : \ \inf_j \phi_{m,j}^E(x_{m,k})<\frac{1}{q}\}.$$  It is straightforward to see that $U_{m,k,q}$ is open in the logic topology on $\frak{M}_n$ and that $F^{-1}(\Ex_n)=\bigcap_{m,k,q}U_{m,k,q}$, so $F^{-1}(\Ex_n)$ is a $G_\de$ subset of $\frak{M}_n$, whence Polish.  We then have that $F|F^{-1}(\Ex_n):F^{-1}(\Ex_n)\to \Ex_n$ is a surjective, continuous, open map, whence $(\Ex_n,\wk)$ is Polish as well by a classical result of Sierpi\'nski (see \cite[Theorem 8.19]{Kec}).  It follows that $\Ex_n$ is a $G_\de$ subset of $\OS_n$ in the weak topology. 
\end{proof}

\begin{nrmk}
By considering $n\geq 5$ and applying the operator system version of Theorem \ref{notGde}, we obtain the operator system version of Theorem \ref{notomittingtypes}.
\end{nrmk}


\begin{nrmk}
Recall from the introduction that Theorem \ref{notomittingtypes} was inspired by a question of Iljias Farah, namely whether or not the class of exact C$^*$-algebras is uniformly definable by a sequence of types.  (For a C$^*$-algebra, exact is the same as 1-exact.)  We point out that Theorem \ref{notomittingtypes} does not seem to lead to a resolution of Farah's question.  Since the forgetful map from the category of C$^*$-algebras to the category of operator systems is an equivalence of categories, it follows from the Beth definability theorem that there are formulae in the language of operator systems that approximate, uniformly over all C$^*$-algebras, the algebra multiplication, whence any types that could be used to define the exact C$^*$-algebras could be taken in the language of operator systems.  Still, this seems to be of little use due to the loose connection between 1-exact operator systems and exact C$^*$-algebras.  For example, there are 1-exact operator systems that do not embed into any exact C$^*$-algebras; see \cite{KW} and \cite{lupini2}. 
\end{nrmk}

\section{Failure of finitary Arveson extension}

In this section, we show how the results of the previous section can be used to show that an approximate, quantitative version of Arveson's extension theorem fails.  


First, it will be convenient to introduce a new definition:

\begin{df}
An operator system $X$ is said to be \emph{CP-rigid} if, for each $n\in \n$ and $\delta>0$, there are $k\in \n$ and $\e>0$ such that, for any unital, self-adjoint map $\phi:E\to X$ with $E$ an $n$-dimensional operator system and $\|\phi\|_k<1+\e$, there is a u.c.p.\ map $\phi':E\to X$ with $\|\phi-\phi'\|<\delta$.
\end{df}

\begin{lemma}\label{CP-rigid}
For any $k$, $M_k$ is CP-rigid.
\end{lemma}

\begin{proof} Let $\phi: E\to M_k$ be as above. By Smith's lemma (see, for example, \cite[Lemma B.4]{BO}), we have that $\|\phi\|_{\cb} = \|\phi\|_k$. Next it follows from the Haagerup--Paulsen--Wittstock 
extension theorem that for any unital, self-adjoint linear map $\phi: E\to M_k$, there is a u.c.p.\ map $\phi': E\to M_k$ with $\|\phi' - \phi\|_{\cb}\leq 2(\|\phi\|_{\cb} - 1)$. (See \cite[Corollary B.9]{BO}.) It follows that, for the parameters $n,\delta$ as above, we may choose $k, \delta/2$.
\end{proof}


We will need an operator space version of the CP-rigidity of matrix algebras:

\begin{cor}\label{CP-rigid-opspace}
For any finite-dimensional operator space $E$ and any $\delta>0$, if $\phi:E\to M_k$ is a linear map with $\|\phi\|_{2k}<1+\delta$, then there is a completely contractive (c.c.) map $\phi':E\to M_k$ with $\|\phi-\phi'\|<2\delta$.
\end{cor}

\begin{proof}
Let $F\subseteq M_2(E)$ be the Paulsen operator system associated to $E$ (see, e.g., \cite[Appendix B]{BO}), so $E=(F)_{12}$.  For $\phi$ as above, we obtain a unital, self-adjoint map $\tilde{\phi}:F\to M_{2k}$ with $(\tilde{\phi})_{12}=\phi$ that also satisfies $\|\tilde{\phi}\|_{2k}<1+\delta$.  By (the proof of) Lemma \ref{CP-rigid}, there is a u.c.p.\ map $\tilde{\psi}:F\to M_{2k}$ such that $\|\tilde{\phi}-\tilde{\psi}\|<2\delta$.  The map $\psi:=(\tilde{\psi})_{12}$ is as desired.
\end{proof}

In \cite[Proposition 2.40]{GS}, it was shown that each $M_k$ has the dual property of \emph{CP-stability}.  As a consequence of this, for every $k$ and $\e>0$, there is $\delta>0$ such that whenever $\phi:M_k\to A$ is a unital map into a C$^*$-algebra for which $\|\phi\|_k\leq 1+\delta$, then there is a u.c.p.\ map $\psi:M_k\to A$ with $\|\phi-\psi\|<\epsilon$.

The following definition contains the central notion of this section.

\begin{df} We say that a sequence of operator systems $(X_n)$ satisfies \emph{finitary Arveson extension} (FAE) if $\prod_\u X_n$ is finitely approximately injective for all nonprincipal ultrafilters $\u$ on $\n$.  If $X$ is an operator system, then we say that $X$ satisfies FAE if the sequence that is constantly $X$ satisfies FAE.
\end{df}

We recall that an operator system $X$ is \emph{approximately injective} (a.i.) if, for every inclusion of finite-dimensional operator systems $E\subset F$, any u.c.p.\ map $\phi: E\to X$ and $\e>0$, there exists a u.c.p.\ map $\phi': F\to X$ so that $\|\phi'|_E - \phi\|<\e$.  $X$ is \emph{finitely approximately injective} (f.a.i.) if the conclusion of approximate injectivity holds with the restriction that $E\subset F\subset M_n$ for some $n$. 

\begin{nrmk}
$X$ is FAE if and only if $X^\u$ is f.a.i.\ for \emph{some} nonprincipal ultrafilter $\u$ on $\n$.  Indeed, suppose that $X^\u$ is f.a.i.\ and consider another nonprincipal ultrafilter $\mathcal{V}$ on $\n$.  Consider finite-dimensional operator systems $E\subseteq F\subseteq M_n$, a u.c.p.\ map $\phi:E\to X^{\mathcal{V}}$, and $\e>0$.  Let $Y$ be a separable elementary substructure of $X^{\mathcal{V}}$ containing $\phi(E)$ and let $\psi:Y\to X^\u$ be an elementary embedding.  Since $X^\u$ is f.a.i.\ , there is a u.c.p.\ map $\chi:F\to X^\u$ such that $\|\chi|E-(\psi\circ \phi)\|<\e$.  Let $Z$ be a separable elementary substructure of $X^\u$ containing $\psi(Y)$ and $\chi(F)$.  Finally, let $\theta:Z\to X^{\mathcal{V}}$ be an elementary embedding such that $\theta|\psi(Y)=\psi^{-1}$.  It follows that $\theta\circ \chi:F\to X^{\mathcal{V}}$ is a u.c.p.\ map for which $\|(\theta\circ \chi)|E-\phi\|<\e$.
\end{nrmk}

\begin{nrmk}
$(M_n)$ is FAE if and only if $\prod_\u M_n$ is f.a.i.\ for \emph{some} nonprincipal ultrafilter $\u$ on $\n$.  Indeed, this follows from the fact any nonprincipal ultraproduct of matrix algebras can be embedded into any other nonprincipal ultraproduct of matrix algebras with conditional expectation.  That being said, we will soon see that $(M_n)$ is not FAE.  
\end{nrmk}


\begin{notation}
Suppose that $(X_i \ : \ i\in I)$ is a family of operator spaces (or operator systems or C$^*$-algebras), $\u$ is a nonprincipal ultrafilter on $I$, and $\prod_\u X_i$ the corresponding ultraproduct.  Given an element $(a_i)\in \ell^\infty(X_i)$, we denote its image in $\prod_\u X_i$ by $(a_i)^\bullet$.  Similarly, if $E$ is an operator space (or operator system or C$^*$-algebra) and $(\phi_i:E\to X_i)$ is a family of uniformly bounded linear maps, we let $(\phi_i)^\bullet:E\to \prod_\u X_i$ be the function defined by $(\phi_i)^\bullet((a_i)^\bullet):=(\phi_i(a_i))^\bullet$.
\end{notation}

The next lemma explains the choice of terminology.
\begin{lemma} \label{finite-fae} 
$(X_n)$ is FAE if and only if the following condition holds: given an inclusion of operator systems $E\subset M_m$ and $k\in \n$, there exists $l\in \n$ so that for any $q\geq l$ and unital map $\phi: E\to X_q$ with $\|\phi\|_l<1+1/l$ there exists a unital map $\phi': M_m\to X_q$ so that $\|\phi'\|_k<1+1/k$ and $\|\phi'|_E -\phi\|<1/k$.
\end{lemma}

\begin{proof} 
First suppose that $(X_n)$ is FAE and fix $E\subseteq M_m$ and $k\in \n$.  Suppose, for a contradiction, that no $l$ exists as desired.  Then for each $l$, we can find $q_l\geq l$ and a unital map $\phi_l:E\to X_{q_l}$ such that $\|\phi_l\|<1+1/l$ and such that $\phi_l$ is at least $1/k$ away from all unital linear maps $\phi':M_m\to X_{q_l}$ with $\|\phi'\|_k<1+1/k$.  Let $I:=\{q_l \ : \ l\in \n\}$ and note that $I$ is infinite.  Let $\u$ be a nonprincipal ultrafilter on $\n$ such that $I\in \u$.  Then $\phi:=(\phi_l)^\bullet:E\to \prod_\u X_n$ is a unital, completely contractive map, whence it is u.c.p.  Since $\prod_\u X_n$ is f.a.i., there is a u.c.p. map $\phi':M_m\to \prod_\u M_n$ such that $\|\phi'|_E-\phi\|<1/k$.  Without loss of generality, we may write $\phi':=(\phi'_n)^\bullet$, where each $\phi'_n:M_m\to X_n$ is unital and linear.  We obtain the desired contradiction since $\|\phi'_n\|_k<1+1/k$ and $\|\phi'_n|_E-\phi_n\|<1/k$ for almost all $n\in I$.  

For the converse, consider a u.c.p. map $\phi:E\to \prod_\u X_n$, where $E\subseteq M_m$ is an operator subsystem, and $\epsilon>0$.  By assumption, for any $k>0$, there is a unital linear map $\phi_k:M_m\to \prod_\u X_n$ such that $\|\phi_k\|_k<1+1/k$ and such that $\|\phi_k|E-\phi\|<\epsilon$.  The result follows from the fact that $\prod_\u X_n$ is $\aleph_1$-saturated.   
\end{proof}


For the rest of this subsection, we fix a nonprincipal ultrafilter $\u$ on $\n$.

\begin{lemma} \label{cb-fae} Assume that $(M_n)$ is FAE.  Fix an inclusion of operator spaces $E\subset M_m$. Then for every $\eta>0$ there exists $l\in \n$ so that for any  map $\phi: E\to \prod_\u M_n$ with $\|\phi\|_l< 1+1/l$, there exists c.c.\ $\phi': M_m\to \prod_\u M_n$ with $\|\phi'|_E - \phi\|<\eta$.
\end{lemma}

\begin{proof} Let $F\subseteq M_2(M_m)$ be the Paulsen operator system associated to $E$.  Let $\delta<\frac{\eta}{2}$ witness the CP-stability of $M_{2m}$ corresponding to the parameter $\frac{\eta}{2}$.  Take $k> \max(2m,\frac{1}{\delta})$.  Let $l\in \n$ be as in the conclusion of Lemma \ref{finite-fae} for $F\subseteq M_{2m}$ and $k$.  

We claim that this $l$ is as desired.  Towards this end, suppose that $\phi:E\to \prod_\u M_n$ is a linear map with $\|\phi\|_l<1+\frac{1}{l}$.  Write $\phi=(\phi_i)^\bullet$ with each $\phi_i:E\to M_i$ linear and $\|\phi_i\|_l<1+\frac{1}{l}$.  We get associated unital, self-adjoint maps $\tilde{\phi_i}:F\to M_2(M_i)$ with $\|\tilde{\phi_i}\|_l< 1+\frac{1}{l}$.  Thus, for $i\geq l$, we get unital maps $\phi_i':M_{2m}\to M_{2i}$ with $\|\phi_i'\|_k\leq 1+\frac{1}{k}$ and $\|\phi_i'|_F-\tilde{\phi_i}\|\leq \frac{1}{k}$.  By choice of $k$, there is a u.c.p.\ map $\psi_i:M_{2m}\to M_{2i}$ with $\|\phi_i'-\psi_i\|\leq \frac{\eta}{2}$.  It follows that $\phi':=((\psi_i)_{12})^\bullet$ is the desired map.   

\end{proof}

Using the previous lemma together with the fact that there are only finitely many subspaces of a given matrix algebra, under the assumption that $(M_n)$ is FAE, there is a function $\sigma:\n\to \n$ so that, given any operator space $E\subseteq M_m$ and any linear map $\phi:E\to \prod_\u M_n$ with $\|\phi\|_{\sigma(m)}\leq 1+\frac{1}{\sigma(m)}$, there is a c.c.\ $\phi':M_m\to \prod_\u M_n$ with $\|\phi'|_E-\phi\|<\frac{1}{m}$.

We are now ready to prove the main result in this section.  


\begin{thm} \label{fae-ot}$(M_n)$ is not FAE.
\end{thm}
\begin{proof} We prove that if $(M_n)$ is FAE, then the class of $1$-exact operator spaces is uniformly definable by a sequence of types in the language of operator spaces, contradicting Theorem \ref{notomittingtypes}.  Towards that end, for every $l,m\in \n$ and every operator space $E$, we define a function $P_{l,m}^E:E^l\to \r$ by setting
\[ P_{l,m}^E(\vec a):=\inf_{\phi,\psi}\left( \|(\psi\circ \phi)(\vec a)-\vec a\|+\frac{1}{m}\right), \] where $\phi:E(\vec a)\to M_m$ is a linear map with $\|\phi\|_{2m}\leq 1+\frac{1}{m}$ and $\psi:\phi(E(\vec a))\to E$ is a linear map with $\|\psi\|_{\sigma(m)}\leq 1+\frac{1}{\sigma(m)}$.  (Here, by $E(\vec a)$ we mean the operator subspace of $E$ generated by $\vec a$.)

We first observe that each $P_{l,m}$ is a definable predicate (relative to the theory of operator spaces).  Indeed, by the discussion in Appendix \ref{defpred}, it suffices to show that:  if $(E_i)$ are operator spaces and $E:=\prod_\u E_i$, then, for every $\vec a=(\vec a_i)^\bullet\in E$, $P_{l,m}^E(\vec a)=\lim_\u P_{l,m}^{E_i}(\vec a_i)$.  We leave it to the reader to verify this equality.

Thus, in order to contradict Theorem \ref{notomittingtypes}, it remains to show that an operator space $E$ is 1-exact if and only if $\inf_mP_{l,m}^E$ is identically $0$ for all $l\in \n$.  First assume that $E$ is $1$-exact and fix $l\in \n$, $\vec a\in E^l$, and $\e>0$.  Since $E$ is $1$-exact, there is $m\in \n$ and c.c.\ maps $\phi:E(\vec a)\to M_m$ and $\psi:\phi(E(\vec a))\to E$ such that $\|(\psi\circ \phi)(\vec a)-\vec a\|<\frac{\e}{2}$.  If $\frac{1}{m}<\frac{\e}{2}$, we would be done.  Nevertheless, one can always remedy the situation to ensure that this is the case.  Indeed, consider the map $\phi_2:E(\vec a)\to M_m\oplus M_m\subseteq M_{2m}$ given by $\phi_2(x)=\phi(x)\oplus \phi(x)$ and the map $\psi_2:\phi(E(\vec a))\oplus \phi(E(\vec a))\to E$ given by $\psi_2(x\oplus y)=(\psi(x)+\psi(y))/2$.  Then $\phi_2$ and $\psi_2$ are still c.c.\ and $\psi_2\circ \phi_2=\psi\circ \phi$.  Thus, one can interate this process until the factorization is through a matrix algebra $M_q$ with $\frac{1}{q}<\frac{\e}{2}$.

We now prove the converse implication.  Suppose that $E$ is an operator space for which $\inf_m P_{l,m}^E$ is identically $0$ for all $l$.  Without loss of generality, $E$ is separable, whence we may further assume that $E$ is concretely represented as an operator subspace of $\prod_\u M_n$.  It suffices to show that the inclusion map $E\hookrightarrow \prod_\u M_n$ is a nuclear embedding.  Towards this end, fix $\vec a\in E^l$ and $\e>0$.  By assumption, there is $m\in \n$, a linear map $\phi:E(\vec a)\to M_m$ with $\|\phi\|_{2m}\leq \frac{1}{m}$ and a linear map $\psi:\phi(E(\vec a))\to E$ with $\|\psi\|_{\sigma(m)}\leq 1+\frac{1}{\sigma(m)}$ for which $\|(\psi\circ \phi)(\vec a)-\vec a\|+\frac{1}{m}<\e$.  By Corollary \ref{CP-rigid-opspace}, there is a c.c.\ map $\phi':E(\vec a)\to M_m$ with $\|\phi-\phi'\|<\frac{2}{m}$.  By the definition of $\sigma$, there is a c.c.\ map $\psi':M_m\to \prod_\u M_n$ with $\|\psi'|_{\phi(E(\vec a))}-\psi\|\leq \frac{1}{m}$.  It remains to observe that $$\|(\psi'\circ \phi')(\vec a)-\vec a\| \leq \frac{2}{m}+\frac{1}{m}(1+\frac{1}{m})+\e \leq 2\e+2\e+\e=5\e.\qedhere$$
\end{proof} 

\begin{cor} \label{prod-not-fi} $\B(H)$ is not FAE.
\end{cor}

\begin{proof} There is an embedding $\prod_\u M_n\hookrightarrow \B(H)^\u$ admitting a conditional expectation; hence, if $\B(H)^\u$ were f.a.i., then this would also be true for $\prod_\u M_n$, a contradiction.
\end{proof}

We end this subsection with some revisionist history.  The argument in the proof of Theorem \ref{fae-ot} was inspired by a very similar argument  of the second author showing that nuclearity for operator systems is uniformly definable by a sequence of types.  The main change that needs to be made is that the analogy of the function $\sigma$ above is guaranteed to exist by the CP-stability of matrix algebras.  (Technically speaking the predicates analogous to the $P_{l,m}$'s above are not actually definable in the language of operator systems-due to usual technicalities concerning unitality-so one needs to work in an expanded language.)  This proof then showed that the nuclear $n$-dimensional operator systems form a weakly $G_\delta$ set (a fact that appears not to have been observed before) and led us to consider whether or not the 1-exact $n$-dimensional operator systems were also weakly $G_\delta$.  The first author later pointed out that the authors of \cite{FHRTW} changed their original proof of the fact that nuclear C$^*$-algebras are uniformly definable by a sequence of types in the language of C$^*$-algebras in such a way that the proof carried over to the category of operator systems. 

\subsection{Clarifying remarks about FAE}

\begin{prop} A unital C$^*$-algebra is f.a.i.\ if and only if it has the WEP.
\end{prop}

\begin{proof}
The proof of this proposition is almost entirely contained in the literature; here we just connect the dots.  Let $A$ denote a unital C$^*$-algebra.  Following \cite{CE}, we say that $A$ is \emph{almost injective} if, for any finite-dimensional operator systems $E\subseteq F$ and any u.c.p.\ map $\phi:E\to A$, there is a u.c.p.\ $\psi:F\to A^{**}$ extending $\phi$.  If we only require this extension property to hold for finite-dimensional matricial operator systems, then $A$ is said to be \emph{finitely almost injective}.  By \cite[Theorem 6.1 and Corollary 6.3]{CE}, we have that $A$ has WEP if and only if $A$ is finitely almost injective.

It thus remains to show that $A$ is f.a.i.\ if and only if it is finitely almost injective.  For the converse, by the duality between c.p.\ maps from $M_n$ to $A^{**}$ and positive operators in $M_n(A^{**})$, given u.c.p.\ $\phi:E\to A$ and a u.c.p. extension $\psi:M_n \to A^{**}$, we can norm approximate $\psi$ by u.c.p.\ maps $\psi': M_n\to A$; see the proof of \cite[Proposition 2.3.8]{BO} for the full details.
\end{proof}

\begin{cor} If $A$ is a C$^*$-algebra, then $A$ satisfies FAE if and only if $A^\u$ has WEP for some (equiv. for all) nonprincipal ultrafilter(s) $\u$ on $\n$.
\end{cor}


\begin{cor}\label{nowep}
$\B(H)^\u$ and $\prod_\u M_n$ do not have the WEP.
\end{cor}

We thank the anonymous referee for pointing out the following corollary.

\begin{cor}\label{subhom} A C$^*$-algebra is FAE if and only if it is subhomogenous.
\end{cor}

We will need the following fact communicated by the referee without proof, which we will now undertake to provide.

\begin{lemma} If $A$ is a unital C$^*$-algebra which is not subhomogenous, then there is a u.c.p.\ embedding $\prod_\u M_n\hookrightarrow A^\u$ with conditional expectation.
\end{lemma}

\begin{proof} Let $n$ be arbitrary. Since $A$ is not subhomogenous, it admits an irreducible representation $\pi: A\to \mathcal B(H)$ with $\dim(H)\geq n$. Let $H_n$ be an $n$-dimensional subspace of $H$ and let $\rho_n: \mathcal B(H)\to \mathcal B(H_n)$ be the canonical conditional expectation given by restriction. By Kadison's transitivity theorem, for any $\e>0$ and any operator $x\in \mathcal B(H_n)\cong M_n$, there is $a\in A$ with $\|\pi(a)\|< \|x\|+\e$ so that $\pi(a)|_{H_n} = x$. It thus follows by the CP-stability of $M_n$ that there is a u.c.p.\ map $\psi_n: M_n\to \pi(A)$ so that $\|\rho_n\circ\psi_n - \id_{M_n}\|< 1/n$. By the Choi-Effros lifting theorem, there is a u.c.p. map $\tilde{\psi}_n:M_n\to A$ such that $\pi\circ \tilde{\psi}_n=\psi_n$.  Let $\psi:=(\tilde{\psi}_n)^\bullet:\prod_\u M_n\to A^\u$.  It follows immediately that $\psi$ is a u.c.p.\ embedding and that the u.c.p.\ map $\rho:=(\rho_n\circ \pi)^\bullet:A^\u\to \prod_\u M_n$ satisfies $\rho\circ \psi=\id_{\prod_\om M_n}$.
\end{proof}

\begin{proof}[Proof of Corollary \ref{subhom}] Let $A$ be a C$^*$-algebra.  First suppose that $A$ is subhomogenous.  Then $A^\u$ is also subhomogeneous, whence nuclear and thus has the WEP. On the other hand, if $A$ is not subhomogenous, then by the previous lemma there is an operator system embedding $\prod_\u M_n\hookrightarrow A^\u$ whose image is the image of a conditional expectation on $A^\u$. Thus if $A^\u$ did have the WEP, it would follow that so would $\prod_\u M_n$, contradicting Corollary \ref{nowep}.
\end{proof}

\begin{nrmk}
The previous corollary has some interesting consequences.  First, it demonstrates that a subalgebra of an FAE C$^*$-algebra is also FAE, a statement which is not immediate from the definition.  Second, we see that for ultrapowers of C$^*$-algebras, nuclearity and WEP are equivalent (and are equivalent to subhomogeneity).  Finally, if $A$ is a QWEP C$^*$-algebra (quotient of a C$^*$-algebra with WEP), then $A^\u$ has Kirchberg's local lifting property, or LLP (see Section \ref{LLP}), if and only if has WEP (if and only if it is nuclear), for QWEP and LLP together imply WEP.
\end{nrmk}







\subsection{Ozawa's argument} Before we had proven that the class of 1-exact operator systems was not uniformly definable by a sequence of types, we had asked Taka Ozawa whether or not $\B(H)$ satisfied FAE. He outlined a proof for us that $\B(H)$ does not satisfy FAE.  Although our proof above is technologically simpler, Ozawa's proof is quite interesting, so we include it here.  We thank him for his permission to include his argument.

There are a few technical preliminaries to dispense with first. Let $R_n$ and $C_n$ be the $n$-dimensional row and column spaces, respectively. Let $RC_n$ be the space $\span\{\delta_i := e_{1i} \oplus e_{i1}\}\subset R_n\oplus C_n \subset M_n\oplus M_n$.\\

The following lemma is essentially contained in Remark 1.2 in \cite{HP}. See also section 2 of \cite{Oz}.

\begin{fact}[Haagerup--Pisier]\label{otherhp}  Let $\phi: RC_n\to X$ be a (completely bounded) map. Suppose $\phi$ extends to a map $\phi': R_n\oplus C_n\to X$ with $\|\phi'\|_{\cb}\leq C$. Setting $x_i := \phi(\delta_i)$, we then have that there exist $a_1,\dotsc,a_n, b_1\dotsc, b_n\in X$ so that $\|\sum a_ia_i^*\|^{1/2},\ \|\sum b_i^* b_i\|^{1/2}\leq C$ and $x_i = a_i + b_i$.
\end{fact}

\begin{proof} Setting $a_i := \phi'(e_{1i}\oplus 0)$ and $b_i := \phi'(0\oplus e_{i1})$, we see that the the required conditions are satisfied.
\end{proof}

\begin{fact}[Haagerup--Pisier]\label{hp} Let $u_1,\dotsc, u_n\in C_r^*(\bF_n)$ be the unitaries corresponding to the standard free generators of $\bF_n$, and let $E_n$ be the operator space spanned by $u_1,\ldots,u_n$. Let $\theta: RC_n\to E_n$ be defined by $\theta(\delta_i) := u_i$. Then $\|\theta\|_{\cb}\leq 2$.
\end{fact}

\begin{proof} See Proposition 1.3 in \cite{HP}.
\end{proof}

\begin{lemma} Assume $\B(H)$ satisfies FAE. Let $E\subset F\subset M_n$ be an inclusion of  operator spaces. Then for any map $\phi: E\to \B(H)^\u$ with $\|\phi\|_{\cb}\leq C$, there is an extension $\phi':F\to \B(H)^\u$ with $\|\phi'\|_{\cb}\leq C$.
\end{lemma}

\begin{proof} Note that $M_2(\B(H)^\om)\cong \B(H)^\u$.
Now, use Paulsen's trick to convert to matricial operator systems and u.c.p.\ maps.
\end{proof}

We now are ready to give Ozawa's proof that $\B(H)$ does not satisfy FAE. Suppose, towards a contradiction, that $\B(H)^\u$ was f.a.i. Since $\B(H)$ is injective, it follows that there exists a conditional expectation $\Phi$ from $\B(H)^\u$ onto $M :=\prod_{\u} M_n(\mathbb{C})$ (C$^*$-algebra ultraproduct). Let $q: M\to M_{\rm vN}$ be the quotient onto the von Neumann ultraproduct. By the remarkable result of Haagerup and Thorbjornsen \cite{HT}, there is a realization $C_r^*(\mathbb F_2)\subset M_{\rm vN}$ which lifts to an embedding $\iota: C_r^*(\mathbb F_2)\to M$. 
Recall that $C_r^*(\mathbb F_\infty)\subset C_r^*(\mathbb F_2)$.  In particular, we have that $(q\circ\iota)(x)=x$ for all $x\in C_r^*(\mathbb F_\infty)$.

Fix arbitrary $n\geq 1$.  Let $u_1,\ldots,u_n$ denote the first free $n$ generators of $C^*_r(\mathbb{F}_\infty)$ and let $\theta_n:RC_n\to \B(H)^\u$ be defined by $\Theta(\delta_i):=\iota(u_i)$.  By Fact \ref{hp}, we have that $\theta_n$ is completely $2$-bounded. Since $\B(H)^\u$ is f.a.i., we have seen that we can extended $\theta_n$ to a completely $2$-bounded map from $R_n\oplus C_n$ into $\B(H)^\u$. By Fact \ref{otherhp}, we have that $\iota(u_i) = a_i + b_i$ with $a_i, b_i\in \B(H)^\u$ and  $\|\sum a_ia_i^*\|^{1/2},\ \|\sum b_i^* b_i\|^{1/2}\leq 2$. Setting $c_i := (q\circ\Phi)(a_i)$ and $d_i := (q\circ\Phi)(b_i)$, we have that $$n = \sum \|u_i\|_2^2\leq \sum 2(\|c_i\|_2^2 + \|d_i\|_2^2)\leq 2(\|\sum a_ia_i^*\| + \|\sum b_i^*b_i\|)\leq 16,$$ yielding a contradiction to the fact that $n$ was arbitrary.

\subsection{Yet another proof that $\B(H)^\u$ is not WEP}

Bradd Hart pointed out to us the following deep theorem of Kirchberg \cite{car} (see also \cite{P}) in light of which the shortest proof that $\B(H)^\u$ does not have WEP became apparent.

\begin{fact}[Kirchberg]\label{braddkirchberg}
Let $M:=\prod_\u M_n$ (the C$^*$ ultraproduct).  Then a C$^*$-algebra $A$ is exact if and only if the norm induced on $A\odot M$ from the tensor product map $A\odot M\hookrightarrow \prod_\u M_n(A)$ is the min-norm
\end{fact}

  Indeed, as mentioned above, if $\B(H)^\u$ had WEP, then $M$ would also have WEP.  If $M$ had WEP, then by Kirchberg's tensorial characterization of WEP, we would have that there is a unique C$^*$ norm on C$^*(\mathbb{F}_\infty)\odot M$.  Combining this fact with Fact \ref{braddkirchberg}, we get that C$^*(\mathbb{F}_\infty)$ is exact, a contradiction.




\section{A Digression on WEP}

Recall from \cite{GS} that a C$^*$-algebra $A$ is said to be \emph{semi-p.e.c.\ as an operator space (resp. system)} if, whenever $B$ is a C$^*$-algebra containing $A$, $\varphi(\vec x)$ is a positive existential formula in the language of operator spaces (resp. systems), and $\vec a$ is a tuple from $A$,  then $\varphi(\vec a)^A=\varphi(\vec a)^B$.  (A positive existential formula is an existential formula whose quantifier-free part is constructed from atomic formulae using only nondecreasing functions as connectives.)  It was proven in \cite{GS} that if a C$^*$-algebra has WEP, then it is semi-p.e.c. as an operator space.  Further, it was shown (Corollary 2.29, op. cit.) that if $A$ is semi-p.e.c.\ as an operator system, then $A$ has the WEP.  It was left as an open question whether or not either of these implications are reversible.  We take the opportunity here to notice that the techniques from \cite{GS} actually show that the latter implication is in fact reversible.

\begin{prop}\label{wep-pec} A unital C$^*$-algebra has the WEP if and only if it is semi-p.e.c.\ as an operator system.
\end{prop}

\begin{proof} We first suppose $A$ is a \emph{separable}, unital C$^*$-algebra with the WEP. Let $A\subset B$ be a unital embedding with $B$ separable.  Let $E_1\subset E_2\subset \dotsc$ be a filtration of $B$ by finite-dimensional operator spaces so that $A\cap (\bigcup_i E_i)$ is dense in $A$.  We now use the proof of \cite[Proposition 2.34]{GS}, specifically the fact the ``WEP implies condition $(\gamma')$'' as defined therein. Thus there are maps $\phi_i: E_i\to E_i\cap A$ with $\|\phi_i\|_i\leq 1$ and $\|\phi_i|_{E_i\cap A} - \id_{E_i\cap A}\|\leq 1/i$.  Given a nonprincipal ultrafilter $\u$ on $\n$, we may define a ``limiting'' map $\phi_\u: \bigcup_i E_i\to A^\u$ which extends completely contractively to a map $\phi: B\to A^\u$ such that $\phi|_A = \id_A$, whence, by unitality, $\phi$ is a u.c.p.\ map. It is now easy to see that $A$ is p.e.c. in $B$.

Now suppose that $A$ is a non-separable unital C$^*$-algebra with the WEP and consider a C$^*$-algebra $B$ containing $A$; we wish to show that $A$ is p.e.c.\ in $B$.  Fix a positive existential formula $\varphi(\vec x)$ in the language of operator systems and a tuple $\vec a$ from $A$.  Let $C$ be a separable elementary substructure of $A$ containing $\vec a$.  By \cite[Corollary 2.2]{GS}, $C$ also has WEP, whence is semi-p.e.c.\ as an operator system by the previous paragraph.  We now have that $\varphi(\vec a)^A=\varphi(\vec a)^C=\varphi(\vec a)^B$.
\end{proof}

\begin{nrmk}
A priori, it seems that $A$ has the WEP if and only if it is p.e.c.\ in $\B(H)$ for every representation $A\hookrightarrow \B(H)$.  In actuality, it suffices to check that $A$ is p.e.c.\ in \emph{some} representation.  Indeed, suppose that the inclusion $A\hookrightarrow \B(H)$ is p.e.c.\ and $A\hookrightarrow \B(K)$ is another representation of $A$.  By Arveson extension, we get a u.c.p.\ map $\B(K)\to \B(H)$ that restricts to the original inclusion $A\hookrightarrow \B(H)$.  It follows immediately that $A$ is p.e.c. in $\B(K)$.
\end{nrmk}

\begin{nrmk}
It had been suggested to us that some variant of the ``linearization trick''  found in \cite{HT} might show that if $A$ is semi-p.e.c.\ as an operator system, then $A$ is p.e.c.\ as a C$^*$-algebra.  The previous proposition can be used to dismiss this possibility.  Indeed, using the fact that simplicity is uniformly definable by a sequence of types of a particularly nice form (see \cite{FHRTW}), one can show that a C$^*$-algebra that is p.e.c.\ must be simple.  As a result, if the aforementioned implication were true, then this would imply that all WEP C$^*$-algebras are simple, which is definitely not the case.
\end{nrmk}

With the equivalence of WEP and semi-p.e.c.\ in the language of operator systems in hand, we now give an elementary proof of one of the main results of \cite{kav-riesz}, namely the equivalence of the WEP and the so-called ``complete tight Riesz interpolation property.'' For the convenience of the reader we recall this notion.

\begin{df}[Kavruk \cite{kav-riesz}] For a unital embedding of C$^*$-subalgebras $A\subset B$, we say that the inclusion has the \emph{complete tight Riesz interpolation property} if for any $n$ and any finite collection $(x_1,\dotsc, x_\ell, y_1, \dotsc, y_m)$ of self-adjoint elements of $M_n(A)$, if there is $z\in M_n(B)$ so that $x_1,\dotsc,x_\ell < z < y_1,\dotsc, y_m$, then there is an element $z'\in M_n(A)$ satisfying the same condition.  (Here, $x<y$ means $y-x\geq \delta\cdot I$ for some $\delta\in \r^{>0}$.) 
\end{df}

The following is Theorem 7.4 from \cite{kav-riesz}.

\begin{thm}[Kavruk] A unital, separable C$^*$-algebra has the WEP if and only if $A$ has the complete tight Riesz interpolation property for some unital inclusion $A\subset \B(H)$. 
\end{thm}

\begin{proof} The forward implication follows immediately from Proposition \ref{wep-pec}.  For the converse, suppose that $A$ has the complete tight Riesz interpolation property.  Consider the positive quantifier-free formula $\varphi(\vec x,\vec y)$, which is necessarily of the form $u(p_1(\vec x_1,\vec y_1),\ldots,p_k(\vec x_k,\vec y_k))$, where $u:\r^k\to \r$ is a non-decreasing continuous function and each $p_i$ is an atomic formula in the language of operator systems. We recall that an atomic formula $p(\vec x, \vec y) = p(x_1,\dotsc, x_\ell, y_1,\dotsc, y_m)$ in the language of operator systems is interpreted in a C$^*$-algebra (or operator system) $A$ as the norm of a $\ast$-linear combination of the unit $1_n\in M_n(A)$ and $x_1,\dotsc,x_\ell, y_1,\dotsc, y_m\in M_n(A)$ for some $n$. Therefore without loss of generality we may assume the variables are self-adjoint. By passing to matrix amplifications, we may also assume that each variable in $\varphi(\vec x, \vec y)$ has the same matrix rank. 

We treat here a special case of the above, the general case being readily deducible from this case using standard operator system and C$^*$-algebraic techniques (for example writing a self-adjoint element $x$ as the difference of two positive elements $x=x_1-x_2$ and using the fact that $\|x\|=\max(\|x_1\|,\|x_2\|)$.) Suppose that $B$ is a C$^*$-algebra containing $A$ and $\vec a$ is a tuple from $M_n(A)_+$ such that $(\inf_{\vec y\in M_n(B)_+}\varphi(\vec a,\vec y))^B=r$.  Fix $\epsilon>0$ and choose $\vec b\in M_n(B)_+$ such that $\varphi(\vec a,\vec b)^B\leq r+\epsilon$.  For $i=1,\ldots,k$, set $r_i:=p_i(\vec a,\vec b)$.  Since $u$ is non-decreasing and continuous, it suffices to find a sufficiently small $\delta>0$ and a tuple $\vec c$ from $M_n(A)_+$ such that $p_i(\vec a,\vec c)\leq r_i+\delta$ for each $i$.
We further suppose that $\vec y$ is a single variable $y$ and that $p_i(\vec x,y)=\|\sum_j \alpha_{i,j}x_j+\beta_i y\|$, where each $\alpha_{i,j}$ and $\beta_i$ are positive real numbers.  Of course, we may also assume that $b\not=0$.  Since $\sum_j \alpha_{i,j}a_j+\beta_i b\leq \|\sum_j \alpha_{i,j}a_j+\beta_i b\|\cdot 1_n=r_i\cdot 1_n$, we have, for each $i$, that $0<b\leq \frac{1}{\beta_i}(r_i-\sum_j \alpha_{i,j}a_j)\cdot 1_n$.  Thus, by the complete tight Riesz property, for any $\delta>0$ sufficiently small, there is $c\in M_n(A)$ such that $0<c<(\frac{1}{\beta_i}(r_i-\sum_j \alpha_{i,j}a_j)+\delta)\cdot 1_n$, whence $\|\sum_j \alpha_{i,j}a_j+\beta_i c\|<r_i+\beta \delta$.
\end{proof}

\section{The spaces $\X_n$ and Kirchberg's Embedding Problem}

In this section, we introduce a variant of the spaces $\OS_n$ and connect them to Kirchberg's Embedding Problem, which asks whether every C$^*$-algebra is embeddable in an ultrapower of the Cuntz algebra $\O_2$.  We let KEP denote the statement that the Kirchberg Embedding Problem has a positive solution.  We refer the reader to our earlier paper \cite{GS} for a comprehensive treatment of this problem.

\begin{df}
For any $n$, we let $\X_n$ denote the space of all $n$-dimensional operator subspaces $E\subset \B(H)$ under the identification between $E, F\subset \B(H)$ if there is a linear bijection $\phi: E\to F$ which extends to a $\ast$-isomorphism from $C^*(E)$ onto $C^*(F)$. 
\end{df}

We now introduce ``weak'' and ``strong'' topologies on $\X_n$ which are analogous to the weak and strong topologies on $\OS_n$. Let $\mathbb P_{n,k}$ be the set of all degree at most $k$, noncommutative $\ast$-polynomials in $n$ variables $p(\vec x) = \sum c_i \vec x^{\alpha(i)}$ with $\sum |c_i|\leq 1$. Given $E, F\in \X_n$ we say that a map $\phi: E\to F$ is \emph{$(m,\de)$-almost multiplicative} if \[\gamma_m(\phi) := \sup_{p\in \mathbb P_{n,m}}\sup_{\vec x\in (E)_1^n} \left|\|p(\vec x)\| -  \|p(\phi(\vec x))\|\right|<\de.\] We set $$\gamma_s(E,F) := \inf_{\phi: E\to F} \sup_m\max\{\gamma_m(\phi), \gamma_m(\phi^{-1})\},$$ where $\phi: E\to F$ ranges over all unital, linear bijections. Finally, we set $$\gamma_w(E,F) = \inf_{\phi:E\to F}\sum_m \frac{1}{2^m} \max\{\gamma_m(\phi), \gamma_m(\phi^{-1})\},$$ where once again $\phi:E\to F$ ranges over all unital, linear bijections.

The following result is proven in the same way as its $\OS_n$ counterpart (see \cite[Proposition 16]{P}).
\begin{prop}\label{mci} A sequence $E_i\to E$ in the weak topology if and only if, for any nonprincipal ultrafilter $\u$ on $\n$, there is a map $
\phi: E\to \prod_\u E_i$ which induces a $\ast$-isomorphism on the corresponding C$^*$-algebras.
\end{prop}

\begin{cor} $(\X_n,\wk)$ is a compact Polish space.
\end{cor}

The proof of the following proposition is straightforward and left to the reader.
\begin{prop} The forgetful map $G: \X_n\to \OS_n$ is weak-weak and strong-strong continuous.
\end{prop}

We need to recall the main result of \cite{GS}.  By a \emph{condition} we mean a finite set $p(\vec x)$ of inequalities of the form $\varphi(\vec x)<\e$, where $\varphi(\vec x)$ is a quantifier-free formula in the language of C$^*$-algebras.  We say that a condition $p(\vec x)$ is \emph{satisfiable} if there is a C$^*$-algebra $A$ and a tuple $\vec a$ from $A$ of the appropriate length for which $\varphi(\vec a)^A<\epsilon$ holds for all inequalities $\varphi(\vec x)<\epsilon$ belonging to $p(\vec x)$; in this case, we say that $\vec a$ \emph{satisfies} $p(\vec x)$.  We say that the (satisfiable) condition $p(\vec x)$ has \emph{good nuclear witnesses} if, for any $\e>0$, there is $k\in \n$ and a tuple $\vec a\in \B(H)^{|\vec x|}$ that satisfies $p(\vec x)$ for which there are u.c.p.\ maps $\phi:S\to M_k$ and $\psi:M_k\to \B(H)$ such that $\|(\psi\circ \phi)(\vec a)-\vec a\|<\epsilon$, where $S$ is the operator system generated by $\vec a$.  (This is not literally the definition of good nuclear witnesses given in \cite{GS}; however,  it is remarked to be equivalent to the original definition after the proof of \cite[Theorem 3.7]{GS}.)

\begin{fact}[Theorem 3.7 in \cite{GS}]\label{gnw} KEP holds if and only if every satisfiable condition has good nuclear witnesses.
\end{fact}

One last bit of notation:  set $\EX_n := G^{-1}(\Ex_n)$.  We are now ready to state our new equivalent reformulation of KEP.

\begin{thm}\label{KEPnew} KEP holds if and only if $\EX_n$ is weakly dense in $\X_n$ for all $n$.
\end{thm}

\begin{proof} 
First suppose that KEP holds and let $E$ be an $n$-dimensional operator space contained in the unital C$^*$-algebra $A$.  Fix a basis $x=x_1,\ldots,x_n$ of $E$ and let $p_i \subset p_{i+1} \subset \dotsb$ be an increasing set of conditions satisfied by $x$ so that $\bigcup_i \phi_i$ is dense in the quantifier-free type of $x$.  By \cite[Proposition 3.5]{GS}, for each $i$, there are exact C$^*$-algebras $A_i$ and tuples $a_i$ from $A$ satisfying each $p_i$.  Without loss of generality, we may assume that each $a_i$ is linearly independent. 
  Let $E_i$ denote the (exact) operator subspace of $A_i$ generated by $a_i$.  Fix a nonprincipal ultrafilter $\u$ on $\n$.  It is then straightforward to see that there is a linear bijection between $E$ and $\prod_\u E_i$ which induces a $\ast$-isomorphism between $C^*(E)$ and $C^*(\prod_\u E_i)$.  It follows that $E_i$ converges weakly to $E$ in $\X_n$.  

In order to prove the converse, by Fact \ref{gnw} it suffices to show that every satisfiable condition has good nuclear witnesses.  Towards that end, let $p(\vec x)$ be a satisfiable condition.  Let $\vec a$ be a tuple from a C$^*$-algebra $A$ that satisfies $p(\vec x)$ and suppose that $E$ is the span of $\vec a$.  If $E$ is $n$-dimensional, then by assumption there exists a sequence $E_i\in \EX_n$ so that $E_i\to E$ weakly.  By Proposition \ref{mci}, there is an $i$ for which $p(\vec x)$ is satisfied by some tuple $\vec a_i$ from $E_i$.  Since $E_i$ is a 1-exact operator system, it follows that $p(\vec x)$ has good nuclear witnesses.

\end{proof}

In light of Theorem \ref{KEPnew} and the fact that $\Ex_n$ is weakly dense in $\OS_n$ for each $n$, a positive answer to the following question would imply that KEP holds:

\begin{question}
Is the forgetful map $G:\X_n\to \OS_n$ open for each $n$?
\end{question}

\section{The Local Lifting Property}\label{LLP}

Recall that a unital C$^*$-algebra $A$ has the \emph{local lifting property} (in the sense of Kirchberg) if, for every unital C$^*$-algebra $C$, every closed ideal $J$ of $C$, every u.c.p.\  map $\phi:A\to C/J$, and every finite-dimensional operator subsystem $X$ of $A$, there is a u.c.p.\  lift $\tilde{\phi}$ of $\phi|X$, that is $\tilde{\phi}:X\to C$ is u.c.p.\ and $q\circ \tilde{\phi}=\phi|X$, where $q:C\to C/J$ is the canonical quotient map.  Replacing $A$ by an operator system, we obtain the notion of the local lifting property for operator systems, denoted osLLP in \cite{kavruk}.

\begin{fact}[Kavruk \cite{kavruk}]\label{kavrukfact}
If $X$ is a finite-dimensional operator system, then $X$ is 1-exact if and only if $X^*$ has osLLP.
\end{fact}

\begin{cor}
The operator systems with osLLP are not uniformly definable by a sequence of types.
\end{cor}

\begin{proof}
As in the proof of Theorem \ref{notomittingtypes} and the remark following it, if the operator systems with osLLP are uniformly definable by a sequence of types, we would have that $\{E\in \OSy_n \ : \ E \text{ has osLLP}\}$ is weakly $G_\delta$.  By taking operator system duals and using Fact \ref{kavrukfact}, we would have that the $1$-exact operator systems would be weakly $G_\delta$, obtaining a contradiction.
\end{proof}

As in the case of exactness, the question of whether or not LLP for C$^*$-algebras is uniformly definable by a sequence of types is open.  In the remainder of this section, we show that LLP is $\la_{\omega_1,\omega}$-axiomatizable in the language of C$^*$-algebras.

For the results stated in this section, it is necessary to make an innocuous addition to the language of C$^*$-algebras, namely we add a sort for $U(A)^\n$, the set of countably infinite sequences of unitaries from $A$.  Since the set of unitaries in a C$^*$-algebra is (0-)definable and taking countably infinite products is part of the construction of the expansion of $A$ by adding imaginaries, this change really is harmless.

\begin{df} We say that a unital C$^*$-algebra $A$ has the \emph{approximate local lifting property} (ALLP) if, for every u.c.p.\ map $\phi: A\to C/J$, every finite-dimensional subspace $E\subset A$, and every $\e>0$ there is a map $\tilde\phi: E\to C$ with $\|\tilde\phi\|_{\cb}<1 + \e$ and satisfying $\|\phi|_E - \pi_J\circ\tilde\phi |_E\|<\e$.
\end{df}

The ALLP seems a priori weaker than LLP; however, thanks to Kirchberg's tensorial characterization of the LLP \cite{kirchberg} they may be seen to coincide.

\begin{prop}\label{LLPreformulations} Let $A$ be a unital, separable C$^*$-algebra. The following statements are equivalent:
\begin{enumerate}
\item  $A$ has the LLP;
\item $A$ has the ALLP;
\item for each finite-dimensional subspace $E\subset A$ and $\e>0$, there exists a $\ast$-homomorphism $\pi: C^*(\f_\infty)\to A$ and a map $\phi: E\to C^*(\f_\infty)$ with $\|\phi\|_{\cb}<1+\e$ and $\|\pi\circ\phi - \id_E\|<\e$;
\item $A\otimes \B(H) = A\otimes_{max} \B(H)$.
\end{enumerate}
\end{prop}

\begin{proof} The implications $(1)\Rightarrow (2)\Rightarrow (3)$ are obvious. The implication $(4)\Leftrightarrow (1)$ is a deep theorem of Kirchberg (\cite[Proposition 1.1(ii)]{kirchberg}). We will demonstrate that $(3)\Rightarrow (4)$.  For brevity of notation, let $C = C^*(\f_\infty)$.

Choose $z = \sum_{i=1}^K z_i\otimes b_i\in A\odot \B(H)$. We will show that $\|z\|_{min} = \|z\|_{max}$. Let $E := \operatorname{span} \{z_1,\dotsc, z_n\}$. By (3), for each $m$, we may fix a $\ast$-homomorphism $\pi_m: C\to A$ and linear map $\phi_m: E\to C$ with $\|\tilde\phi_m\|_{\cb}< 1+ \tfrac{1}{m}$ and so that $\|(\pi_m\circ\phi_m) - \id_E\|< \tfrac{1}{m}$. Since $\cb$-norms are well-behaved with respect to the minimal tensor product (see, for example, \cite[Theorem 3.5.3]{BO}), we have, for $\phi_m\otimes\id : E\otimes \B(H)\to C\otimes \B(H)$, that $\|\phi_m\otimes\id\|_{\cb}\leq 1 + \tfrac{1}{m}$. Since $C$ has LLP (see \cite[Lemma 2.1]{kirchberg}), we have that $C\otimes \B(H) \cong C\otimes_{max} \B(H)$.  By universality of the maximal tensor product, $\pi_m\otimes \id: C\otimes_{max} \B(H)\to A\otimes_{max} \B(H)$ is a $\ast$-homomorphism, whence contractive.  Let $z_m := ((\pi_m\circ\phi_m)\otimes \id)(z)$.  It follows that $\|z_m\|_{max}\leq (1 + \tfrac{1}{m})\|z\|_{min}$. On the other hand, since $\otimes_{max}$ is a cross-norm we have by the triangle inequality that \[\|z - z_m\|_{max}\leq K\max_i \|b_i\| \max_j \|(\pi_m\circ\phi_m)(z_j) -z_j\|\leq \tfrac{K}{m} \max_i \|b_i\|,\] and we are done.
\end{proof}

The proof of the following corollary is analogous to the proof of \cite[Corollary 3.9]{qwep} using Proposition \ref{LLPreformulations}; it is not needed in the sequel.

\begin{cor} Let $A$ be a separable C$^*$-algebra, and let $\theta: C^*(\bF_\infty)\to A$ be a $\ast$-homomorphism which sends the canonical unitary generators onto a dense set of unitaries. Then $A$ has the LLP if and only if the identity map on $A$ has an approximate local c.b.\ lift through $\theta$.
\end{cor}

Let $\cS_n\subset C^*(\f_n)$ be the operator system spanned by the canonical unitaries and their inverses.  Let $\cS_{n,k} \subset C^*(\f_n)$ be the operator system which is the $k$-fold product of $\cS_{n,1} := \cS_n$ with itself, i.e., $\cS_{n,k}$ is the linear span of $\{x_1\dotsb x_k : x_i\in \cS_n, i=1,\dotsc,k\}$. 


We refer the reader to \cite{BI} for a treatment of infinitary continuous logic.

\begin{prop}
The LLP is $\la_{\omega_1,\omega}$-axiomatizable in the language of C$^*$-algebras.
\end{prop}

\begin{proof}
Let $A$ be a unital C$^*$-algebra. For an $n$-tuple $\vec a\in A^n$, let $E(\vec a) := \operatorname{span} \{a_1,\dotsc,a_n\}$ and
let $\Phi_{n,m,k}(\vec a)$ denote the set of linear maps $\phi: E(\vec a)\to \cS_{m,k}$.  For $\vec v\in U(A)^\n$, let $\pi_{\vec v}:C^*(\f_\infty)\to A$ denote the $\ast$-homomorphism determined by mapping the canonical unitary $u_i$ of C$^*(\f_\infty)$ to $v_i$ and let $$\chi_{n,m,k,\ell}(\vec v,\vec a):=\inf_{\phi \in \Phi_{n,m,k}(\vec a)}\max\{\|\phi\|_\ell -1, \|(\pi_{\vec v}\circ\phi)(\vec a) - \vec a\|\}.$$

By Beth definability (and the fact that $\cS_{m,k}$ is finite-dimensional), $\chi_{n,m,k,l}$ is a definable predicate in the language of C$^*$-algebras.  Further note that, for a fixed $n\in \n$, all of the formulas $\chi_{n,m,k,\ell}$ have the same modulus of uniform continuity.  Now consider the $\la_{\omega_1,\omega}$-sentence $$\psi_n:=\sup_{\vec x}\inf_{\vec v\in U(A)^{\n}}\inf_{m,k}\sup_\ell\chi_{n,m,k,\ell}(\vec v,\vec x).$$


We claim that $A$ has the LLP if and only if $\psi_n^A=0$ for all $n$. If $A$ has the LLP, it is obvious that $\psi_n^A=0$ for all $n$ by the small perturbation argument as $\bigcup_{m,k} \cS_{m,k}$ is dense in $C^*(\f_\infty)$.

On the other hand, suppose that $\psi_n^A=0$ for each $n$.  We show that $A$ satisfies condition (3) of Proposition \ref{LLPreformulations}. By assumption, for any tuple $\vec a$ and $\e>0$ there exists $v\in U(A)^\n$, $m,k$ and a sequence of maps $\phi_\ell: E(\vec a)\to \cS_{m,k}$ with $\|\phi_\ell\|_\ell<1+\e$ and $\|(\pi_{\vec v}\circ\phi_\ell)(\vec a) - \vec a\|<1+\e$. Now define \[\phi := (\phi_\ell)^\bullet: E(\vec a)\to \cS_{m,k}^\omega \cong \cS_{m,k}.\] It follows that $\|\phi\|_{\cb}\leq 1+\e$ and $\|(\pi_{\vec v}\circ\phi)(\vec a) - \vec a\|\leq \e$.
\end{proof}

Examining the proof of $\la_{\omega_1,\omega}$-axiomatizability for the LLP, it seems natural to formulate a weakening of the LLP:  we say that $A$ had the \emph{weak local lifiting property} (WLLP) if, for each finite-dimensional subspace $E\subset A$ and $k,\e>0$, there exists a $\ast$-homomorphism $\pi: C^*(\f_\infty)\to A$ and a map $\phi: E\to C^*(\f_\infty)$ with $\|\phi\|_{k}<1+\e$ and $\|\pi\circ\phi - \id_E\|<\e$. However, it is a result of Robertson and Smith \cite[Proposition 2.4]{RS}, that \emph{every} C$^*$-quotient map $q: B\to B/J$ admits $n$-positive local liftings for all $n$, whence every C$^*$-algebra has the WLLP.  (We thank the anonymous referee for pointing this fact out.)




We point out a potentially interesting consequence:

\begin{prop}
Let $A$ be a unital, separable C$^*$-algebra.  Then $A$ has LLP if and only if $A$ is CP-stable.
\end{prop}

\begin{proof} As just pointed out above, every C$^*$-algebra has the WLLP. It is clear that CP-stability and the WLLP together imply the ALLP, hence the LLP.  Conversely, it is proven in \cite{GS} that CP-stability is equivalent to yet another weakening of LLP, namely the \emph{local ultrapower lifting property} (LULP), where one only requires local lifts for maps into ultrapowers (which are, in fact, particular instances of quotients). 
\end{proof}

The preceding proposition lends itself to a much simpler proof of \cite[Corollary 2.43]{GS}, which states that there is a separable C$^*$-algebra that is not contained in a CP-stable C$^*$-algebra.  Indeed, if $A$ is a WEP C$^*$-algebra that does not have LLP (which exists by the main result of \cite{jungepisier}), then $A$ cannot be contained in an LLP C$^*$-algebra.

As mentioned above, in \cite{GS} it is shown that CP-stability is equivalent to the local ultrapower lifting property, whence, by the previous proposition, we say that the LLP is equivalent to the local ultrapoower lifting property.  We now want to point out that the LLP is equivalent to an a priori weaker property.

\begin{df}
Let $A$ be a unital C$^*$-algebra.  We say that $A$ has the \emph{local matrix ultraproduct lifting property} if, for any u.c.p.\ map $\phi:A\to \prod_\u M_n$ and any finite-dimensional operator system $E\subseteq A$, we have a u.c.p.\ lift of $\phi|_E$.
\end{df}

\begin{prop}
Let $A$ be a unital, separable C$^*$-algebra.  Then $A$ has LLP if and only if $A$ has the local matrix ultraproduct lifting property.
\end{prop}

\begin{proof}
Suppose that $A$ has the local matrix ultraproduct lifting property.  It suffices to show that $A$ is CP-stable.  Towards that end, fix a finite-dimensional operator system $E\subseteq A$ and $\epsilon>0$; we need to find $k\in \n$ and a finite-dimensional operator system $E\subset E'\subset A$ so that if $\phi:E'\to B$ is a unital map into a unital C$^*$-algebra with $\|\phi\|_k<1+1/k$, then there is a u.c.p.\ map $\phi':E\to B$ such that $\|\phi |_E-\phi'\|<\epsilon$.  Fix $\delta>0$ sufficiently small. For a finite-dimensional operator system $A\subset E\subset E'$, let $I_k(E')$ denote the set of $n$ for which, whenever $\tau:E'\to M_n$ is a unital map with $\|\tau\|_k<1+1/k$, then there is a u.c.p.\ map $\tau':E\to M_n$ with $\|\tau|_E-\tau'\|<\delta$.  We leave it to the reader to check (along the same lines as \cite[Proposition 2.42]{GS}) that the fact that $A$ has the local matrix ultraproduct lifting property implies that there are $k$ and $E'$ for which $I_k(E')\in \u$.  We claim that these $k$ and $E'$ are as desired.  

Let $\phi:E'\to B$ be a unital map into a unital C$^*$-algebra with $\|\phi\|_k<1+1/k$.  Let $F\subseteq B$ be the operator system generated by $\phi(E')$.  Since $F$ admits a complete order embedding into $\prod_\u M_n$, we obtain unital maps $\psi_q:F\to M_{n_q}$ such that $\|\psi_q\|_q,\|\psi_q^{-1}\|_q<1+1/q$.   For $q\in I_k(E')$, we can find u.c.p.\ maps $\tau_q:E\to M_{n_q}$ which are $\delta$ close to $\psi_q\circ \phi |_E$.  By the small perturbation argument, for these $q$, we can find unital, self-adjoint maps $\tau_q':E\to \psi_q(F)$ with $\|\tau_q'\|_{\cb}\leq 1+N\delta$ for some $N\in \n$ depending only on $E$.  Taking a cluster point of the maps $\psi_q^{-1}\circ \tau_q':E\to F$, we obtain a unital, self-adjoint map $\theta:E\to F$ which is $N^2\delta$ close to $\phi$ with $\|\theta'\|_{\cb}\leq 1+N\delta$.  Finally, by \cite[Corollary B.11]{BO}, we can find u.c.p.\ $\phi':E\to B$ which is $10N^2\delta$ close to $\theta$.  It follows that by choosing $\delta$ small enough (depending only on $E$ and $\epsilon$), $\phi'$ is within $\epsilon$ of $\phi$.
\end{proof}

\begin{nrmk}
The Robertson and Smith result alluded to above shows that every C$^*$-algebra admits a complete order embedding into an ultrapower of C$^*(\mathbb{F}_\infty)$.  As a result, this shows that C$^*(\mathbb{F}_\infty)$ is not CP-rigid.  Indeed, if C$^*(\mathbb{F}_\infty)$ were CP-rigid, then every C$^*$-algebra $A$ would be operator space finitely representable in C$^*(\mathbb{F}_\infty)$ (that is, for any finite-dimensional operator subspace $E$ of $A$ and any $\epsilon>0$, there is a finite-dimensional operator subspace $\hat{E}$ of C$^*(\mathbb{F}_\infty)$ with $d_{\cb}(E,\hat{E})<\epsilon$), which for C$^*$-algebras with WEP is shown in \cite{jungepisier} to be equivalent to LLP.  Thus, if C$^*(\mathbb{F}_\infty)$ were CP-rigid, then WEP would imply LLP, contradicting the main result of \cite{jungepisier}.
\end{nrmk}

Returning to the logical status of LLP, we have yet to determine whether or not LLP is uniformly definable by a sequence of types.  We end this paper with motivation for settling this question.  

Let the statement ``LLP is omitting types'' be an abbreviation for the statement that the collection of C$^*$-algebras with LLP be uniformly definable by a sequence of \emph{universal} types in the language of C$^*$-algebras; here, the requirement of the types being universal signifies that the formulae $\phi_{m,j}$ are existential.  (This reversal makes sense, for asking that a collection of existential formulae be bounded below is a universal statement.)  Since LLP is preserved under existentially closed substructures, it is reasonable to expect that the types defining LLP (should they exist) are universal.

Recall that KEP stands for the statement that every C$^*$-algebra embeds into an ultrapower of $\O_2$.  Say that a C$^*$-algebra $A$ is \emph{locally universal} if every C$^*$-algebra embeds into an ultrapower of $A$.  By abstract nonsense (see \cite{FHS}), locally universal C$^*$-algebras exist.  Using this terminology (and the fact that every separable nuclear C$^*$-algebra embeds into $\O_2$) we see that KEP is equivalent to the existence of a locally universal nuclear C$^*$-algebra.  It thus makes sense to consider the LLPEP, which is the statement that there exists a locally universal C$^*$-algebra with LLP.  (The corresponding statement for WEP is automatically true for one can just embed a locally universal C$^*$-algebra into a C$^*$-algebra with WEP.)  Clearly KEP implies LLPEP, but it is unclear whether or not the two statements are equivalent.

One final acronym:  by the \emph{weak QWEP conjecture} we mean the statement that there is a non-nuclear C$^*$-algebra that has both WEP and LLP.  (This conjecture appears as a question at the end of Ozawa's article \cite{qwep}.)  The reason for the name is that Kirchberg's QWEP conjecture states that every C$^*$-algebra is QWEP, which by the results in \cite{kirchberg}, is equivalent to the statement that the LLP implies the WEP (which is itself equivalent to Connes' Embedding Problem having a positive solution).

\begin{prop}
Suppose that LLP is omitting types.  Then either KEP is equivalent to LLPEP or else the weak QWEP conjecture is true.
\end{prop}

\begin{proof}
Using the machinery of model-theoretic forcing as in \cite{GS}, the LLPEP implies that there exists an existentially closed C$^*$-algebra $A$ with LLP.  Either $A$ is nuclear and KEP holds (see \cite{GS}) or else $A$ is non-nuclear.  Since existentially closed C$^*$-algebras have WEP, the latter alternative witnesses the truth of the weak QWEP conjecture.
\end{proof}

\appendix

\section{The Polish space of codes for structures}\label{codes}

In this appendix, we review the notion of codes for structures.  This topic is discussed \cite[Section 1]{BNT} for relational languages.  We present the topic here for the sake of completeness and also for arbitrary languages.

Fix a continuous, separable language $\la$ (not necessarily relational!).  For the sake of exposition, we assume that $\la$ is 1-sorted, although we will eventually apply this discussion to the language of operator spaces or the language of operator systems, both of which are many sorted.  

Fix a countable  set $(\varphi_i)_{i\in \n}$ of atomic $\la$-formulae that is dense in the set of all atomic formulae in any finite number of variables.  Set $n_i$ to be the arity of $\varphi_i$.  We will think of $P\in \prod_i \mathbb{R}^{\n^{n_i}}$ as potentially coding an $\la$-pre-structure $\mathcal{M}$ with $\n$ as a distinguished countable dense set according to the rule $\varphi_i(\vec n)^{\mathcal M}:=P(i)(\vec n)$.  Of course, not all elements of $\prod_i \r^{\n_{n_i}}$ represents such codes.  First, we note that we can read off $d(m,n)$ for any $m,n\in \n$ by setting $d(m,n):=\lim_{j\to \infty} \varphi_{i_j}(m,n)$, where $\varphi_{i_j}(x,y)$ converges uniformly to $d(x,y)$.  We will say that $P$ \emph{codes a pseudometric space} if the induced function $d:\n^2\to \r$ is a pseudometric.    We will say that $P$ \emph{codes an $\la$-prestructure} if $P$ codes a pseudometric space and $P$ respects the modulus of uniform continuity for each $\varphi_i$, that is, for any tuples $\vec m,\vec n$ of the appropriate length such that $d(\vec m,\vec n)<\Delta_{\varphi_i}(\epsilon)$, we have $|P(i)(\vec m)-P(i)(\vec n)|\leq \e$.  We let $\frak{M}$ denote the set of codes of $\la$-prestructures.  It is quite easy to see that $\frak{M}$ is a $G_\de$ subset of $\prod_i \r^{\n^{n_i}}$, whence is Polish.  We refer to this Polish topology on $\frak{M}$ as the \emph{logic topology} on the space of codes.

Given $P\in \frak{M}$, we can construct an $\la$-structure $\mathcal{M}_P$ as follows.  First, one lets $Y_P$ denote the so-called \emph{term algebra} on $\n$, that is, all expressions one obtains from $\n$ by successive applications of the function symbols.  Note that the pseudometric $d_P$ extends naturally to $Y_P$ as $d(t_1,t_2)$ can be read off from the dense sequence of formulae. Moreover, $Y_P$ is still an $\la$-prestructure.  One can then separate and complete $Y_P$ to obtain an $\la$-structure $\mathcal{M}_P$; see \cite[Section 3]{bbhu}  


We will need the following well-known fact (which is also straightforward to verify from the definitions).

\begin{lemma}\label{closed}
Suppose that $T$ is a universal theory.  Let $\frak{M}_T$ denote the elements of $\frak{M}$ that code models of $T$.  Then $\frak{M}_T$ is closed in the logic topology.
\end{lemma}


\section{Definable predicates}\label{defpred}

In this appendix, we review the necessary background on definable predicates.  For more explanation and proofs, see \cite[Section 9]{bbhu} and \cite{FHRTW}.  

Let $T$ be an $\la$-theory.  Let $\mathcal{F}_n$ denote the set of $\la$-formulae in $n$ free variables.  There is a natural seminorm $\|\cdot\|_T$ on $\mathcal{F}_n$ given by $\|\varphi\|_T:=\sup\{\varphi^{\mathcal M}(\vec a) \ : \ \mathcal{M}\models T, \vec a\in M^n\}$.  One can separate and complete $(\mathcal{F}_n,\|\cdot\|_T)$; the elements of this completion are called \emph{$n$-ary definable predicates} for $T$ and the set of $n$-ary definable predicates for $T$ forms a Banach algebra.  Given any $n$-ary definable predicate $P$, there are $\la$-formulae $(\varphi_m)$ from $\mathcal{F}_n$ such that $\|P-\varphi_m\|_T\leq \frac{1}{m}$.  This allows us to define, for any $\mathcal{M}\models T$, a uniformly continuous function $P^{\mathcal M}:M^n\to \r$ by $P^{\mathcal M}(\vec a):=\lim_{m\to \infty} \varphi^{\mathcal M}_m(\vec a)$.   

Conversely, suppose that there is an association $\mathcal{M}\mapsto P^{\mathcal{M}}$ from models $\mathcal{M}$ of $T$ to uniformly continuous functions $P^{\mathcal{M}}:M^n\to \r$ where the modulus of uniform continuity is the same for all $P^{\mathcal M}$'s.  Then Beth's definability theorem implies that $P$ is a definable predicate for $T$ if and only if the class of structures $(\mathcal{M},P^{\mathcal{M}})$ is an axiomatizable class.  Moreover, in order to check that this latter property holds, one needs to check that, for any  models $\mathcal{M}_i$ of $T$ with ultraproduct $\mathcal{M}:=\prod_\u M_i$, we have $P^{\mathcal M}=\lim_\u P^{\mathcal{M}_i}$.

The following lemma shows us that we can use definable predicates when showing that a property is uniformly definable by a sequence of types; the distinction between formulae and definable predicates is safely ignored in \cite{FHRTW}, but we include a proof here for the sake of completeness.

\begin{lemma}
If we use definable predicates rather than formulae in the definition of uniformly definable by a sequence of types, we do not get a different notion.
\end{lemma}

\begin{proof}
Suppose that, for each $m,j$, $P_{m,j}(\vec x_m)$ is a nonnegative definable predicate and $(\varphi_{m,j}^n)_n$ are nonnegative formulae such that $\|P_{m,j}-\varphi_{m,j}^n\|_\infty<\frac{1}{n}$.  Then the following two requirements are equivalent:
\begin{itemize}
\item $\sup_{\vec x_m} \inf_j P_{m,j}(\vec x_m)=0$ for all $m$;
\item $\sup_{\vec x_m} \inf_j (\varphi_{m,j}^n(\vec x_m)\dotminus \frac{1}{n})=0$ for all $m$ and $n$.  
\end{itemize}
Note also that the modulus of uniform continuity for $\varphi_{m,j}^n\dotminus \frac{1}{n}$ depends only on $m$ and $n$.
\end{proof}


\end{document}